\documentclass[11pt]{article}

\usepackage[T1]{fontenc}
\usepackage[utf8]{inputenc}
\usepackage{lmodern, microtype}
\usepackage{latexsym, textcomp}
\usepackage{amsmath, amssymb, amsthm}
\usepackage{array, setspace, verbatim}

\usepackage[english]{babel}
\usepackage[english]{varioref}

\usepackage[pdftex]{hyperref}
\usepackage{pdfpages}

\usepackage[a4paper]{geometry}

\newcommand{\mathstyle}[2]{%
  \ifthenelse{\equal{#1}{ord}}{\mathord{#2}}{%
    \ifthenelse{\equal{#1}{op}}{\mathop{#2}}{%
      \ifthenelse{\equal{#1}{bin}}{\mathbin{#2}}{%
        \ifthenelse{\equal{#1}{rel}}{\mathrel{#2}}{%
          \ifthenelse{\equal{#1}{open}}{\mathopen{#2}}{%
            \ifthenelse{\equal{#1}{close}}{\mathclose{#2}}{%
              \ifthenelse{\equal{#1}{punct}}{\mathpunct{#2}}{%
                \ifthenelse{\equal{#1}{inner}}{\mathinner{#2}}{#2}%
              }%
            }%
          }%
        }%
      }%
    }%
  }%
}

\newcommand*{\nidt}{\noindent}
\newcommand*{\dst}{\displaystyle}
\newcommand*{\fns}{\footnotesize}

\newcommand*{\ov}[1]{\overline{#1}}

\newcommand*{\wt}[1]{\widetilde{#1}}

\newcommand*{\espf}{\vspace*{1ex}}

\newcommand*{\ii}{\infty}
\newcommand*{\del}{\nabla}
\renewcommand*{\epsilon}{\varepsilon}
\renewcommand*{\phi}{\varphi}

\newcommand*{\C}{\mathbb{C}}
\newcommand*{\Cc}{\mathcal{C}}
\newcommand*{\D}{\mathbb{D}}
\renewcommand*{\H}{\mathbb{H}}
\newcommand*{\Mm}{\mathcal{M}}
\newcommand*{\R}{\mathbb{R}}

\renewcommand*{\S}{\mathbb{S}}
\newcommand*{\Ss}{\mathcal{S}}
\newcommand*{\Uu}{\mathcal{U}}

\newcommand*{\arXiv}{\text{arXiv}}
\newcommand*{\Nil}{\text{Nil}}

\newcommand*{\lp}{\left(}
\newcommand*{\rp}{\right)}
\newcommand*{\lc}{\left[}
\newcommand*{\rc}{\right]}
\newcommand*{\lac}{\left\{}
\newcommand*{\rac}{\right\}}
\newcommand*{\lan}{\langle}
\newcommand*{\ran}{\rangle}
\newcommand*{\lb}{\left|}
\newcommand*{\rb}{\right|}
\newcommand*{\lr}{\left.}
\newcommand*{\rr}{\right.}

\renewcommand*{\leq}{\leqslant}
\renewcommand*{\geq}{\geqslant}

\renewcommand*{\a}{\forall}
\newcommand*{\into}{\rightarrow}

\newcommand*{\st}[1]{%
  \ifthenelse{\equal{#1}{}}{\mathrel{|}}{%
    \ifthenelse{\equal{#1}{big}}{\mathrel{\big{|}}}{%
      \ifthenelse{\equal{#1}{Big}}{\mathrel{\Big{|}}}{%
        \ifthenelse{\equal{#1}{bigg}}{\mathrel{\bigg{|}}}{%
          \ifthenelse{\equal{#1}{Bigg}}{\mathrel{\Bigg{|}}}
        }%
      }%
    }%
  }
}

\newcommand*{\Lim}[2]{\lim_{#1\into#2}}
\renewcommand*{\matrix}[2]{\lp\begin{array}{@{}*{#1}{c}@{}}#2\end{array}\rp}
\newcommand*{\der}[2][]{\frac{\partial#1}{\partial#2}}

\theoremstyle{definition}
\newtheorem{definition}{Definition}[section]
\newtheorem{remark}[definition]{Remark}
\theoremstyle{plain}
\newtheorem{lemma}[definition]{Lemma}
\newtheorem{proposition}[definition]{Proposition}
\newtheorem{theorem}[definition]{Theorem}

\newenvironment{introtheorem}[2][]{\par%
  \bfseries\trivlist%
  \item[\hskip\labelsep{\ifthenelse{\equal{#1}{}}{Theorem}{#1}}~\ref{#2}.]%
  \normalfont\itshape\ignorespaces%
}{\endtrivlist}
\renewenvironment{proof}[1][]{\par%
  \normalfont\trivlist%
  \item[\hskip\labelsep{\itshape Proof}\ifx#1\else{~(#1)}\fi.]%
  \ignorespaces%
}{\hfill~${\square}$\endtrivlist}

\hypersetup{%
pdfauthor= {S{\'e}bastien Cartier},%
pdftitle= {Saddle towers in Heisenberg space},%
pdfcreator= {PDFLaTeX},%
pdfproducer= {PDFLaTeX}%
}

\title{Saddle towers in Heisenberg space}
\author{S{\'e}bastien Cartier}

\begin{document}

\maketitle

\begin{abstract}
\nidt We construct most symmetric Saddle towers in Heisenberg space i.e. periodic minimal surfaces that can be seen as the desingularization of vertical planes intersecting equiangularly. The key point is the construction of a suitable barrier to ensure the convergence of a family of bounded minimal disks. Such a barrier is actually a periodic deformation of a minimal plane with prescribed asymptotic behavior. A consequence of the barrier construction is that the number of disjoint minimal graphs suppoerted on domains is not bounded in Heisenberg space.
\end{abstract}

\nidt \textit{Mathematics Subject Classification:} \emph{53A10, 53C42}.

\section{Introduction}

The study of periodic surfaces has recently encountered new developments in homogeneous $3$-spaces. In $\H^2 \times \R$, new examples have been constructed such as for instance doubly periodic minimal surfaces by Mazet, Rodr{\'i}guez and Rosenberg~\cite{MaRoRo} or genus one constant mean curvature $1/ 2$ surfaces by Plehnert~\cite{Pl}. Other examples are involved in important results such as the resolution of Alexandrov problem in a quotient space of $\H^2 \times \R$ by Menezes~\cite{Men} or the fact the Calabi-Yau conjectures do not hold for embedded minimal surface in $\H^2 \times \R$ by Rodr{\'i}guez and Tinaglia~\cite{RoTi}. In Heisenberg space, new examples are mostly of Scherk type like the Jenkins-Serrin theorem for compact domains obtained by Pinheiro~\cite{Pi}.

The present paper deals with the construction of another kind of Scherk type surfaces, called \emph{most symmetric Saddle towers}, in Heisenberg space. A Saddle tower is a minimal surface that can be thought of as the desingularization of $n$ vertical planes, $n \geq 2$, intersecting along a vertical geodesic; in particular it is a complete embedded minimal surface with $2n$ planar ends. Historically, they were first found by Scherk~\cite{Sc} for $n= 2$ in Euclidean space $\R^3$ ---~a Saddle tower with $n= 2$ is usually called a \emph{singly periodic Scherk surface}~--- and (more than) a century and half later Karcher~\cite{Kar} generalized the construction to any $n \geq 2$. Another couple of decades later, Morabito and Rodr{\'i}guez~\cite{MoRo} and Pyo~\cite{Py} constructed Saddle towers in $\H^2 \times \R$. Recently, Menezes~\cite{Men3} constructed the singly (and doubly) periodic Scherk surfaces in semi-direct product spaces, including Heisenberg space.

\medskip

We use a classical method to construct Saddle towers in Heisenberg space (see Section~\ref{sec:saddle}). We first construct a sequence of embedded minimal disks ---~by solving Plateau problems on suitable Jordan curves~--- such that it converges to an embedded minimal surface bordered by horizontal geodesic arcs, which we call the \emph{fundamental piece}. Reflecting the fundamental piece along the geodesics in its boundary, we get the desired Saddle tower:

\begin{introtheorem}{thm:saddle}
For any $n \geq 2$ and $a> 0$, there exists a properly embedded singly periodic minimal surface of genus zero with $2n$ planar ends distributed at constant angle $\pi/ n$. We call it a \emph{most symmetric Saddle tower}.
\end{introtheorem}

Since the Saddle tower is composed of reflections of the fundamental piece, its ends must be equiangularly distributed to ensure the embeddedness of the surface. Note that Morabito and Rodr{\'i}guez~\cite{MoRo} constructed non equiangular Saddle towers with a conjugate Plateau technique, which does not apply in Heisenberg space.

The key step is to prove the convergence of the sequence of minimal disks, which is done by using a suitable barrier. We construct that barrier in Section~\ref{sec:barrier} by deforming a minimal plane so that we control the asymptotic behavior of the deformed minimal surfaces. The deformation technique is based on the existence of an elliptic operator, called the \emph{compactified mean curvature operator}, containing information on both the asymptotic behavior and the mean curvature. This technique has already been used by the author and Hauswirth~\cite{CaHa} for surfaces in $\H^2 \times \R$ of constant mean curvature $1/ 2$.

It is an immediate consequence of Weierstrass representation that the construction of such surfaces is not possible in Euclidean space since ---~under reasonable asumptions on the total curvature~--- minimal ends are asymptotically rotational. The present barrier construction highlights another difference with Euclidean case:

\begin{introtheorem}{thm:disjgraphs}
In Heisenberg space, the number of disjoint domains supporting minimal vertical graphs is unbounded.
\end{introtheorem}

\section{Preliminaries}

We recall basic definitions and properties on Heisenberg space and fix classic notations.

\subsection{A model for Heisenberg space}

$3$-dimensional Heisenberg space $\Nil_3$ is a nilpotent Lie group which is usually represented in $\text{GL}_3(\R)$ by:
\[
\Nil_3= \lac \lr \matrix{3}{1 & a & c \\ 0 & 1 & b \\ 0 & 0 & 1}~\rb~a, b, c \in \R \rac,
\]
and its Lie algebra $\mathfrak{nil}_3$ can be seen as the subset of $\Mm_3(\R)$:
\[
\mathfrak{nil}_3= \lac \lr \matrix{3}{0 & a & c \\ 0 & 0 & b \\ 0 & 0 & 0}~\rb~a, b, c \in \R \rac
\]
In this paper we use the parametrization of $\Nil_3$ induced by the exponential map:
\[
\text{exp}: \matrix{3}{0 & x_1 & x_3 \\ 0 & 0 & x_2 \\ 0 & 0 & 0} \in \mathfrak{nil}_3 \mapsto \matrix{3}{1 & x_1 & \dst x_3+ \frac{x_1x_2}{2} \\ 0 & 1 & x_2 \\ 0 & 0 & 1} \in \Nil_3.
\]
With these notations, $\Nil_3$ identifies with $\R^3$ and the group law is:
\[
(x_1, x_2, x_3)*(y_1, y_2, y_3)= \lp x_1+ y_1, x_2+ y_2, x_3+ y_3+ \frac{x_1y_2- x_2y_1}{2} \rp.
\]
We call \emph{canonical frame} the frame $(E_1, E_2, E_3)$ which is the extension by left translation of the canonical frame of $\R^3$ at the origin i.e.:
\[
E_1= \der{x_1}- \frac{x_2}{2} \der{x_3}, \quad E_2= \der{x_2}+ \frac{x_1}{2} \der{x_3} \quad \text{and} \quad E_3= \der{x_3}.
\]
Endowing $\Nil_3$ with the following left-invariant metric:
\[
\lan \cdot, \cdot \ran= dx_1^2+ dx_2^2+ \lp \frac{1}{2} (x_2dx_1- x_1dx_2)+ dx_3 \rp^2,
\]
makes the canonical frame orthonormal.

Heisenberg space is a riemannian fibration for the projection on the first two coordinates $\pi: (x_1, x_2, x_3) \in \Nil_3 \mapsto (x_1, x_2) \in \R^2$, which means that the isomorphism $d\pi|_{(\ker d\pi)^{\bot}}$ onto the Euclidean plane is an isometry. Given this structure, vector fields spanned by $(E_1, E_2)$ are referred to as \emph{horizontal} and the direction $E_3$ is said \emph{vertical}. In particular, a \emph{vertical graph} in $\Nil_3$ is a complete immersion transverse to $E_3$.

In the sequel, we mostly work in cylindrical coordinates $(\rho, \theta, x_3)$ with $\rho \geq 0$ and $\theta \in \R$ such that:
\[
x_1= \rho \cos \theta \quad \text{and} \quad x_2= \rho \sin \theta,
\]
and we consider the \emph{cylindrical frame} $(E_{\rho}, E_{\theta}, E_3)$ with:
\begin{gather*}
E_{\rho}= \cos \theta E_1+ \sin \theta E_2= \cos \theta \der{x_1}+ \sin \theta \der{x_2} \\
\text{and} \quad E_{\theta}= -\sin \theta E_1+ \cos \theta E_2= -\sin \theta \der{x_1}+ \cos \theta \der{x_2}+ \frac{\rho}{2} \der{x_3}.
\end{gather*}
The cylindrical frame is also orthonormal and the Levi-Civita connection $\del$ on $\Nil_3$ in terms of $(E_{\rho}, E_{\theta}, E_3)$ is given by:
\[
\begin{array}{lll}
\del_{E_{\rho}} E_{\rho}= 0                       & \dst \del_{E_{\theta}} E_{\rho}= \frac{1}{\rho} E_{\theta}- \frac{1}{2} E_3 & \dst \del_{E_3} E_{\rho}= -\frac{1}{2} E_{\theta} \espf \\
\dst \del_{E_{\rho}} E_{\theta}= \frac{1}{2} E_3  & \dst \del_{E_{\theta}} E_{\theta}= -\frac{1}{\rho} E_{\rho}                 & \dst \del_{E_3} E_{\theta}= \frac{1}{2} E_{\rho}  \espf \\
\dst \del_{E_{\rho}} E_3= -\frac{1}{2} E_{\theta} & \dst \del_{E_{\theta}} E_3= \frac{1}{2} E_{\rho}                            & \del_{E_3} E_3= 0
\end{array}
\]

\subsection{Model surface}

The barrier constructed in Section~\ref{sec:barrier} comes from a (suitable) deformation of a surface, labeled $S^0$, to which we refer as the \emph{model surface}. In our model of $\Nil_3$, $S^0$ is nothing but the entire graph $\lac x_3= 0 \rac$. It is a complete embedded minimal surface, its tangent plane is spanned by the coordinate vector fields:
\[
\der{x_1}= E_1+ \frac{x_2}{2} E_3 \quad \text{and} \quad \der{x_2}= E_2- \frac{x_1}{2} E_3,
\]
and has unit normal:
\[
N^0= \frac{1}{\sqrt{4+ \rho^2}} (\rho E_{\theta}+ 2 E_3).
\]
The end of $S^0$ is of annulus type and is vertical in the sense that the normal vector $N^0$ is asymptotically horizontal: $\lan N^0, E_3 \ran \into 0$ when $\rho \into +\ii$.

\medskip

Consider a point $x^0= (x^0_1, x^0_2, x^0_3)$ in $\Nil_3$ and a unit vector $V$ in the tangent space $T_{x^0} \Nil_3$ at $x^0$, written $V= R\cos \phi E_1+ R\sin \phi E_2+ \gamma E_3$ with $R \geq 0$, $\phi, \gamma \in \R$ and $R^2+ \gamma^2= 1$. The geodesic passing through $x^0$ and directed by $V$ at $x^0$ admits the parametrization $t \in \R \mapsto \big{(} x_1(t), x_2(t), x_3(t) \big{)}$ by arc length with (see~\cite{Mar} for details):
\[
\lac \begin{array}{l}
\dst x_1(t)= x^0_1+ \frac{R}{2\gamma} \big{(} \sin(2\gamma t+ \phi)- \sin \phi \big{)} \espf\\
\dst x_2(t)= x^0_2- \frac{R}{2\gamma} \big{(} \cos(2\gamma t+ \phi)- \cos \phi \big{)} \espf\\
\dst x_3(t)= x^0_3- \frac{R}{4\gamma} \Big{[} x^0_1 \big{(} \cos(2\gamma t+ \phi)- \cos \phi \big{)}+ x^0_2 \big{(} \sin(2\gamma t+ \phi)- \sin \phi \big{)} \Big{]} \espf\\
\dst \hspace{22em}+ \frac{1+ \gamma^2}{2\gamma} t- \frac{R^2}{4\gamma^2} \sin(2\gamma t)
\end{array} \rr.
\]If $x^0 \in S^0$ and $V= N^0$, meaning:
\[
R= \frac{\rho}{\sqrt{4+ \rho^2}}, \quad \phi= \theta+ \frac{\pi}{2} \quad \text{and} \quad \gamma= \frac{2}{\sqrt{4+ \rho^2}},
\]
with $x^0_1= \rho \cos\theta$, $x^0_2= \rho \sin\theta$ and $x^0_3= 0$, the parametrization writes:
\begin{equation} \label{eq:equidist}
\lac \begin{array}{l}
\dst x_1(t)= \rho \cos\theta+ \frac{\rho}{4} \lc \cos \lp \frac{4t}{\sqrt{4+ \rho^2}}+ \theta \rp- \cos \theta \rc \espf\\
\dst x_2(t)= \rho \sin\theta+ \frac{\rho}{4} \lc \sin \lp \frac{4t}{\sqrt{4+ \rho^2}}+ \theta \rp- \sin \theta \rc \espf\\
\dst x_3(t)= \frac{\rho^2}{16} \sin \lp \frac{4t}{\sqrt{4+ \rho^2}} \rp+ \frac{8+ \rho^2}{16} \frac{4t}{\sqrt{4+ \rho^2}} \espf
\end{array} \rr.
\end{equation}
Fixing $t \neq 0$ and letting $(\rho, \theta)$ vary, System~\eqref{eq:equidist} becomes a parametrization of the equidistant surface to $S^0$ at signed (normal) distance $t$. And making $\rho \into +\ii$, we see that asymptotically the equidistant surface grows linearly with respect to $\rho$ ---~namely, the equidistant surface asymptotically behaves like the quadric:
\[
t^2 x_1^2+ t^2 x_2^2- 4x_3^2= t^2 \lp 4- \frac{t^2}{3} \rp.
\]

The growth of equidistant surfaces to the model surface is important in Section~\ref{sec:barrier}. Indeed, the information we have on the asymptotic behavior of the deformations of $S^0$ is precisely the asymptotic normal signed distance to $S^0$. 

\begin{definition} \label{def:asympdist}
Consider a surface $S$ which can be parametrized by:
\[
(\rho, \theta) \mapsto \big{(} \rho \cos\theta, \rho \sin\theta, h(\rho, \theta) \big{)},
\]
at least for $\rho$ big enough and some $\theta \in \R$. When it exists, the \emph{asymptotic horizontal signed distance} of $S$ to the model surface $S^0$ in the direction $\theta$ is the quantity:
\[
d_{\ii}(S, S^0)(\theta)= \Lim{\rho}{+\ii} \frac{2}{\rho} h(\rho, \theta);
\]
the qualification ``horizontal'' coming from the fact that $N^0$ is asymptotically horizontal.
\end{definition}

\subsection{Schwarz symmetrization}

The classical Schwarz reflection for minimal surfaces in Euclidean space states that if a minimal immersion contains a straight line in its boundary, it can be smoothly extended across that line by the $\pi$-rotation about the line. This result applies to Heisenberg space ---~and in general to homogeneous $3$-spaces with $4$-dimensional isometry group~--- if restricted to extensions across horizontal or vertical geodesics (see~\cite{AbRo2} and details in~\cite{Rosen, MaTo}). Moreover, in Heisenberg space the geodesic reflection along a horizontal geodesic passing through the vertical axis $\lac x_1= x_2= 0 \rac$ is exactly the euclidean $\pi$-rotation about the geodesic.

Consider a minimal surface $\Sigma$ bordered only by horizontal and vertical geodesic arcs. Then $\Sigma$ can be extended to a smooth minimal surface $\hat{\Sigma}$ by geodesic reflections along each geodesic in the boundary of $\Sigma$. The boundary of $\hat{\Sigma}$ also only contains horizontal and vertical geodesic arcs and the extension process can be iterated. Note that a meeting point of several geodesic arcs of the boundary of $\Sigma$ is a removable singularity of the extended surface $\hat{\Sigma}$ and thus $\hat{\Sigma}$ is smooth at this point (see~\cite{ChSc, Gu}).

We use Schwarz symmetrization in Section~\ref{sec:saddle} to extend the fundamental piece to a complete embedded smooth 	minimal surface with the desired symmetry properties.

\subsection{Some notations}

Let $\D= \lac z \in \C \st{big} |z|< 1 \rac$ be the open unit disk, $\ov{\D}= \lac z \in \C \st{big} |z| \leq 1 \rac$ its closure and $(r, \theta)$ the polar coordinates on $\ov{\D}$. The boundary $\partial\D$ of $\D$ is identified with $\S^1$.

The space $\Cc^{k, \alpha}(\ov{\D})$, with $k \geq 0$ and $0< \alpha< 1$, is the usual H{\"o}lder space over $\ov{\D}$ and $\Cc^{k, \alpha}_0(\ov{\D})$ is the subspace of $\Cc^{k, \alpha}(\ov{\D})$ of functions that are zero on the boundary of $\D$.

Finally, we consider the spaces $L^2(\cdot)$ endowed with the natural scalar product denoted $\lan \cdot, \cdot \ran_{L^2(\cdot)}$ and Hilbert norm $|\cdot|_{L^2(\cdot)}$.

\section{Barrier construction} \label{sec:barrier}

The model surface $S^0$ is an embedded disk. We deform it by mean of a differential operator taking into account both the asymptotic behavior and the mean curvature. The construction of this operator is called a \emph{compactification} of the mean curvature since it is based on a conformal change of the induced metric to extend it on the boundary of the parametrizing disk. To do so, we need a conformal parametrization of $S^0$ which writes (in polar coordinates):
\[
X^0: (r, \theta) \in \D \mapsto \lp \frac{4r}{1- r^2} \cos\theta,  \frac{4r}{1- r^2} \sin\theta, 0 \rp.
\]
We parametrize entire vertical graphs of $\Nil_3$ by:
\[
X^{\eta}: (r, \theta) \in \D \mapsto \lp \frac{4r}{1- r^2} \cos\theta,  \frac{4r}{1- r^2} \sin\theta, \eta(r, \theta) \frac{1+ r^2}{1- r^2} \rp,
\]
for some map $\eta: \D \into \R$. We call such a parametrization \emph{graph coordinates at infinity}.

In the sequel, we are interested in graphs such that $\eta \in \Cc^{2, \alpha}(\ov{\D})$ and we use graph coordinates at infinity to compactify surfaces and quantify their asymptotic behavior. Indeed, using Definition~\ref{def:asympdist} a graph $X^{\eta}$ with $\eta \in \Cc^{2, \alpha}(\ov{\D})$ is at asymptotic horizontal signed distance to the model surface $d_{\ii}(X^{\eta}, S^0)(\theta)= \eta(1, \theta)$. Surfaces are thus considered as \emph{compact surfaces with boundary} and we can apply the method first developed by White \cite{Wh}.Also note that the value $\eta|_{\partial \D}$ is invariant under vertical translations.

\subsection{Compactification of the mean curvature}

From now on, to ease the notations, we denote with indexes $1, 2$ quantities related to coordinates $r, \theta$ respectively. 

\begin{theorem} \label{thm:compcourb}
For any entire graph admitting graph coordinates $X^{\eta}$ with $\eta \in \Cc^{2, \alpha}(\ov{\D})$, the mean curvature $H(\eta)$ verifies:
\begin{equation} \label{eq:comph}
\frac{2}{r} \sqrt{|g(0)|} H(\eta)= \sum_{i, j} A_{ij}(r, \eta, D\eta) \eta_{ij}+ B(r, \eta, D\eta),
\end{equation}
where $|g(0)|$ is the determinant of the metric induced by $X^0$, $A_{ij}$ and $B$ are $\Cc^{0, \alpha}$ functions on $\ov{\D}$ which are real-analytic in their variables and $A= (A_{ij})$ is a coercive matrix on $\ov{\D}$.
\end{theorem}

\begin{proof}[See Appendix~\ref{sec:proof} for computation details]
The definition of the mean curvature of $X^{\eta}$ is the following:
\[
2H(\eta)= \sum_{i, j} g^{ij}(\eta) \lan \del_{X^{\eta}_i} X^{\eta}_j, N^{\eta} \ran,
\]
where $X^{\eta}_1$ (resp. $X^{\eta}_2$) is the derivative of $X^{\eta}$ with respect to $r$ (resp. $\theta$), $N^{\eta}$ is the unit normal to $X^{\eta}$ and $\big{(} g^{ij}(\eta) \big{)}$ is the inverse matrix of the metric $g(\eta)= \big{(} g_{ij}(\eta) \big{)}$ induced by $X^{\eta}$. Namely, we have:
\begin{gather*}
X^{\eta}_1(r, \theta)= \frac{4(1+ r^2)}{(1- r^2)^2} E_{\rho}+ \frac{4r}{(1- r^2)^2} \lp \eta+ \frac{(1+ r^2)\eta_1}{4r} (1- r^2) \rp E_3 \\
\text{and} \quad X^{\eta}_2(r, \theta)= \frac{4r}{1- r^2} E_{\theta}- \frac{8r^2}{(1- r^2)^2} \lp 1- \frac{(1+ r^2)\eta_2}{8r^2} (1- r^2) \rp E_3,
\end{gather*}
with $(E_{\rho}, E_{\theta}, E_3)$ the cylindrical frame. The first fundamental form $g(\eta)$ is:
\begin{gather*}
g_{11}(\eta)= \frac{16(1+ r^2)^2}{(1- r^2)^4} \lc \lp 1+ \frac{r^2 \eta^2}{(1+ r^2)^2} \rp+ \frac{r\eta \eta_1}{2(1+ r^2)} (1- r^2)+ \frac{\eta_1^2}{16} (1- r^2)^2 \rc, \\
g_{12}(\eta)= -\frac{32r^3}{(1- r^2)^4} \lc \eta+ \frac{1+ r^2}{4r} \lp \eta_1- \frac{\eta \eta_2}{2r} \rp (1- r^2)- \frac{(1+ r^2)^2 \eta_1 \eta_2}{32r^3} (1- r^2)^2 \rc \\
\text{and} \quad g_{22}(\eta)= \frac{16r^2(1+ r^2)^2}{(1- r^2)^4} \lc 1- \frac{\eta_2}{1+ r^2} (1- r^2)+ \frac{\eta_2^2}{16r^2} (1- r^2)^2 \rc,
\end{gather*}
and its determinant $|g(\eta)|$ writes:
\[
|g(\eta)|= \lp \frac{16r(1+ r^2)^2}{(1- r^2)^4} w(\eta) \rp^2= |g(0)|w^2(\eta),
\]
with the following expression of $w(\eta)$:
\begin{multline*}
w(\eta)= \Bigg{[} 1- \frac{\eta_2}{1+ r^2} (1- r^2)+ \lp \frac{r^2 \eta^2}{(1+ r^2)^4}+ \frac{\eta_2^2}{16} \rp (1- r^2)^2+ \frac{r\eta \eta_1}{2(1+ r^2)^3} (1- r^2)^3 \\
+\frac{\eta_1^2}{16(1+ r^2)^2} (1- r^2)^4 \Bigg{]}^{1/ 2}.
\end{multline*}
Also the expression of the unit normal $N^{\eta}$ is:
\begin{multline*}
N^{\eta}= \frac{1}{\sqrt{|g(\eta)|}} X^{\eta}_1 \times X^{\eta}_2= \frac{1}{w(\eta)} \Bigg{[} -\frac{r(1- r^2)}{(1+ r^2)^2} \lp \eta+ \frac{(1+ r^2)\eta_1}{4r} (1- r^2) \rp E_{\rho} \\
+\frac{2r}{1+ r^2} \lp 1- \frac{(1+ r^2) \eta_2}{8r^2} (1- r^2) \rp E_{\theta}+ \frac{1- r^2}{1+ r^2} E_3 \Bigg{]}.
\end{multline*}

The computation detailed in Appendix~\ref{sec:proof} gives the expression~\eqref{eq:comph} with the desired regularity and:
\[
A_{11}= 1+ O(1- r^2), \quad A_{12}= A_{21}= \frac{\eta}{2}+ O(1- r^2) \quad \text{and} \quad A_{22}= 1+ \frac{\eta^2}{4}+ O(1- r^2),
\]
which shows that $A$ is coercive on $\ov{\D}$.
\end{proof}

We call such a process a \emph{compactification} of the mean curvature since the quantity $r^{-1}\sqrt{g(0)}H(\eta)$ can be extended to the boundary $\partial\D$. It is also strongly linked with the compactification of the induced metric $g(\eta)$ by the following equality:
\[
A^{-1}= \matrix{2}{1+ \eta^2/ 4 & \eta/ 2 \\ \eta/ 2 & 1}+ O(1- r^2)= \frac{1}{\sqrt{|g(0)|}} g(\eta)+ O(1- r^2).
\]

\medskip

From now on, we denote $H$ the operator:
\[
H: \eta \in \Cc^{2, \alpha}(\ov{\D}) \mapsto H(\eta) \in \Cc^{0, \alpha}(\ov{\D}),
\]
where $H(\eta)$ is the mean curvature of $X^{\eta}$, and we call it the \emph{mean curvature operator}. Using Theorem~\ref{thm:compcourb}, we define the \emph{compactified mean curvature operator} to be:
\[
\ov{H}: \eta \in \Cc^{2, \alpha}(\ov{\D}) \mapsto \frac{2}{r} \sqrt{|g(0)|} H(\eta) \in \Cc^{0, \alpha}(\ov{\D}).
\]

The \emph{compactified Jacobi operator} is $\ov{L}= D\ov{H}(0): \Cc^{2, \alpha}(\ov{\D}) \into \Cc^{0, \alpha}(\ov{\D})$. We get that:
\[
\ov{L}= \frac{\sqrt{|g(0)|}}{r} L,
\]
where $L$ is the Jacobi operator of $S^0$, since it is a standard fact that:
\[
\a \eta \in \Cc^{2, \alpha}(\ov{\D}),\ DH(0) \cdot \eta= \frac{1}{2} L\eta.
\]
Furthermore, conducting the computation in Appendix~\ref{sec:proof} more carefully, we get an explicit expression of $\ov{L}$:
\[
\a \eta \in \Cc^{2, \alpha}(\ov{\D}),\ \ov{L}\eta= \Delta\eta+ \frac{8\eta}{(1+ r^2)^2},
\]
where $\Delta$ stands for the usual flat laplacian.

We highlight two immediate consequences. First, the vertical coordinate $\phi^0$ of $N^0$ ---~meaning $\phi^0= \lan N^0, E_3 \ran$ ~--- is in the kernel of $\ov{L}$ i.e. $\ov{L}\phi^0= 0$. And second, $\ov{L}$ verifies the following Green identity:
\begin{equation} \label{eq:green}
\a u, v \in \Cc^{2, \alpha}(\ov{\D}),\ \int_{\ov{\D}} \lp u\ov{L} v- v\ov{L} u \rp d\ov{A}= \int_0^{2\pi} \lr \lp u \der[v]{r}- v \der[u]{r} \rp \rb_{r= 1} d\theta,
\end{equation}
with $d\ov{A}$ the Lebesgue measure on $\ov{\D}$. We deduce the following lemma:

\begin{lemma} \label{lem:nosol}
There is no solution $u \in \Cc^{2, \alpha}(\ov{\D})$ to the equation:
\[
\lac \begin{array}{ll}
\ov{L} u= 0         & \text{on } \ov{\D} \\
u|_{\partial \D}= 1
\end{array} \rr.
\]
\end{lemma}

\begin{proof}
By contradiction, suppose such a $u$ exist and apply the Green identity~\eqref{eq:green} to $\phi^0$ and $u$:
\begin{align*}
0 & = \int_{\ov{\D}} \lp \phi^0 \ov{L} u- u\ov{L} \phi^0 \rp d\ov{A}= \int_0^{2\pi} \lr \lp \phi^0 \der[u]{r}- u \der[\phi^0]{r} \rp \rb_{r= 1} d\theta \\
  & = \int_0^{2\pi} d\theta= 2\pi,
\end{align*}
since:
\[
\phi^0|_{r= 1}= \lr \frac{1- r^2}{1+ r^2} \rb_{r= 1}= 0 \quad \text{and} \quad \lr \der[\phi^0]{r} \rb_{r= 1}= \lr -\frac{4r}{(1+ r^2)^2} \rb_{r= 1}= -1.
\]
This is impossible.
\end{proof}

Consider the restriction $\ov{L}_0$ of $\ov{L}$ to $\Cc^{2, \alpha}_0(\ov{\D})$ and $K= \ker \ov{L}_0$. We use the inclusions $\Cc^{2, \alpha}_0(\ov{\D}) \subset \Cc^{0, \alpha}(\ov{\D}) \subset L^2(\D)$ and denote $K^{\bot}$ the orthogonal to $K$ in $\Cc^{0, \alpha}(\ov{\D})$ for the natural scalar product of $L^2(\D)$ and $K_0^{\bot}= K^{\bot} \cap \Cc^{2, \alpha}_0(\ov{\D})$.

A standard result states that the restriction $\ov{L}_0$ is a Fredholm operator with index zero (see \cite{GiTr}) and furthermore $K= \R \phi^0$ and $\ov{L}_0 \big{(} \Cc^{2, \alpha}_0(\ov{\D}) \big{)}= K^{\bot}$.

\subsection{Deformations of the model surface}

Let $\mu: \Cc^{2, \alpha}(\S^1) \into \Cc^{2, \alpha}(\ov{\D})$ be the operator such that $\mu(\gamma)$ is the harmonic function on $\ov{\D}$ (for the flat laplacian) with value $\gamma$ on the boundary $\partial \D$. In the sequel, we decompose $\Cc^{2, \alpha}(\ov{\D})$ into $\Cc^{2, \alpha}(\S^1) \times \R \times K_0^{\bot}$, meaning that any $\eta \in \Cc^{2, \alpha}(\ov{\D})$ is in one-to-one correspondence with a triple $(\gamma, \lambda, \sigma) \in \Cc^{2, \alpha}(\S^1) \times \R \times K_0^{\bot}$ such that:
\[
\eta= \mu(\gamma)+ \lambda \phi^0+ \sigma.
\]

Consider $\Pi_K$ and $\Pi_{K^{\bot}}$, the orthogonal projections on $K$ and $K^{\bot}$ respectively. We follow White~\cite{Wh} and define a suitable map to apply the Implicit Function Theorem:

\begin{lemma}
Let $\Phi: \Cc^{2, \alpha}(\S^1) \times \R \times K_0^{\bot} \into K^{\bot}$ be the map defined by:
\[
\Phi(\gamma, \lambda, \sigma)= \Pi_{K^{\bot}} \circ \ov{H} \big{(} \mu_a(\gamma)+ \lambda \phi^a+ \sigma \big{)}.
\]
Then $D_3 \Phi(\gamma^a, 0, 0): K_0^{\bot} \into K^{\bot}$ is an isomorphism.
\end{lemma}

\begin{proof}
We compute that $D_3 \Phi(\gamma^a, 0, 0)= \Pi_{K^{\bot}} \circ \ov{L}_0|_{K_0^{\bot}}$ and we know $K^{\bot}$ is the range of $\ov{L}_0$, which means $D_3 \Phi(\gamma^a, 0, 0): K_0^{\bot} \into K^{\bot}$ is an isomorphism.
\end{proof}

Therefore, we can apply the Implicit Function Theorem to $\Phi$, which states that there exist an open neighborhood $U$ of $(0, 0)$ in $\Cc^{2, \alpha}(\S^1) \times \R$ and a unique smooth map $\sigma: U \into K_0^{\bot}$ such that:
\[
\a (\gamma, \lambda) \in U,\ \Phi\big{(} \gamma, \lambda, \sigma(\gamma, \lambda) \big{)}= 0.
\]
Consequently, define smooth maps $\eta: U \into \Cc^{2, \alpha}(\ov{\D})$ and $\kappa: U \into K$ by:
\begin{gather*}
\eta(\gamma, \lambda)= \mu(\gamma)+ \lambda \phi^0+ \sigma(\gamma, \lambda) \quad \text{and} \quad \kappa(\gamma, \lambda)= \Pi_K \circ \ov{H} \big{(} \eta(\gamma, \lambda) \big{)}.
\end{gather*}
If an entire graph admits $X^{\eta(\gamma, \lambda)}$ as graph coordinates at infinity, we say that $\lac \gamma, \lambda \rac$ are the \emph{data} of the surface and that $\gamma$ is the \emph{value at infinity}.

\begin{lemma} \label{lem:eta}
The map $\eta$ has the following properties:
\begin{enumerate}
	\item\label{pt:zero} $\eta(0, 0)= 0$.
	\item\label{pt:valbound} $\a (\gamma, \lambda) \in U,\ \eta(\gamma, \lambda)|_{\partial \D}= \gamma$.
	\item\label{pt:diff} $D_2\eta(0, 0): \lambda \in \R \mapsto \lambda \phi^0 \in \Cc^{2, \alpha}(\ov{\D})$.
\end{enumerate}
\end{lemma}

\begin{proof}
Point~\ref{pt:zero} is a direct consequence of the uniqueness in the Implicit Function Theorem. For Point~\ref{pt:valbound} compute:
\[
\eta(\gamma, \lambda)|_{\partial \D}= \mu(\gamma)|_{\partial \D}+ \lambda \phi^0|_{\partial \D}+ \sigma(\gamma, \lambda)|_{\partial \D}= \gamma.
\]
And for Point~\ref{pt:diff}, we only have to show $D_2 \sigma(\gamma^a, 0)= 0$. To do so, compute:
\begin{align*}
0 & = \lr \frac{d}{dt} \rb_{t= 0} \Phi \big{(} 0, t, \sigma(0, t) \big{)}= \Pi_{K^{\bot}} \circ \ov{L} \big{(} \phi^0+ D_2 \sigma(0, 0) \cdot 1 \big{)} \\
  & = \Pi_{K^{\bot}} \circ \ov{L}_0 \big{(} \phi^0+ D_2 \sigma(0, 0) \cdot 1 \big{)}= \Pi_{K^{\bot}} \circ \ov{L}_0 \big{(} D_2 \sigma(0, 0) \cdot 1 \big{)} \\
  & = \ov{L}_0 \big{(} D_2 \sigma(0, 0) \cdot 1 \big{)}.
\end{align*}
Thus, $D_2 \sigma(0, 0) \cdot 1 \in K \cap K_0^{\bot}= \lac 0 \rac$ i.e. $D_2 \sigma(0, 0)= 0$.
\end{proof}

\begin{remark} \label{rmk:lambda}
Lemma~\ref{lem:eta} Point~\ref{pt:valbound} shows that the value at infinity of a surface $X^{\eta(\gamma, \lambda)}$ does not depend on $\lambda$, meaning that given a value at infinity $\gamma$ there exists a $1$-parameter family of surfaces all with value at infinity equals to $\gamma$. Moreover, it can be shown that any two surfaces in this family are congruent only up to a vertical translation. Indeed, from the half-space theorem of Daniel, Meeks and Rosenberg~\cite{DaMeRo}, we know that the difference of heights of two entire minimal graphs diverges unless the graphs differ from each other by a vertical translation. In our case, if $(\gamma, \lambda), (\gamma, \lambda') \in U$ then:
\[
\eta(\gamma, \lambda)- \eta(\gamma, \lambda')= O(1- r^2),
\]
and the difference of heights is bounded.
\end{remark}

The values of the mean curvature of deformations $X^{\eta(\gamma, \lambda)}$ of $S^0$ are determined by $\kappa$. Indeed, for $(\gamma, \lambda) \in U$, we have $\Phi\big{(} \gamma, \lambda, \sigma(\gamma, \lambda) \big{)}= 0$ and:
\begin{equation} \label{eq:hkappa}
\ov{H} \big{(} \eta(\gamma, \lambda) \big{)}= \kappa(\gamma, \lambda)+ \Phi\big{(} \gamma, \lambda, \sigma(\gamma, \lambda) \big{)}= \kappa(\gamma, \lambda).
\end{equation}
Consider $\Uu= \kappa^{-1}(\lac 0 \rac) \cap U$. The fact that the parameter $\lambda$ is associated to vertical translations means that we can write $\Uu= \Gamma \times \R$, with $\Gamma$ a subset of $\Cc^{2, \alpha}(\S^1)$. Furthermore, since the construction is local, we can suppose $\Gamma$ connected.

\begin{proposition} \label{prop:local}
$\Gamma$ is a codimension $1$ smooth submanifold of $\Cc^{2, \alpha}(\S^1)$ is a subset of:
\[
\lac \gamma \in \Cc^{2, \alpha}(\S^1) \lb \int_0^{2\pi} \gamma\,d\theta= 0 \rr \rac.
\]
\end{proposition}

\begin{proof}
We show that $\kappa$ is a submersion at $(0, 0)$. From Equation~\eqref{eq:hkappa} compute:
\begin{align*}
D_2 \kappa(0, 0) \cdot 1 & = \lr \frac{d}{dt} \rb_{t= 0} \kappa(0, t)= \lr \frac{d}{dt} \rb_{t= 0} \ov{H} \big{(} \eta(0, t) \big{)} \\
                         & = \ov{L} \big{(} D_2 \eta(0, 0) \cdot 1 \big{)}= \ov{L}_0(\phi^0)= 0,
\end{align*}
using Lemma~\ref{lem:eta} Point~\ref{pt:diff} and that $\phi^0 \in K$. Moreover, $D_1\kappa(0, 0) \cdot 1$ is not identically zero. Indeed, using \eqref{eq:hkappa}:
\begin{align*}
D_1 \kappa(0, 0) \cdot 1 & = \lr \frac{d}{dt} \rb_{t= 0} \kappa(t, 0)= \lr \frac{d}{dt} \rb_{t= 0} \ov{H} \big{(} \eta(t, 0) \big{)} \\
                         & = \ov{L} \big{(} D_1 \eta(0, 0) \cdot 1 \big{)} \neq 0,
\end{align*}
using Corollary~\ref{lem:nosol} with $\big{(} D_1 \eta(0, 0) \cdot 1 \big{)}|_{\partial \D}= 1$ deduced from Lemma~\ref{lem:eta} Point~\ref{pt:valbound}. Since $D\kappa$ is continuous and non zero at $(0, 0)$, there exists an open neighborhood of $(0, 0)$ in $\Cc^{2, \alpha}(\S^1) \times \R$ on which $\kappa$ is a submersion. Therefore, up to a restriction on $\Gamma$, we can suppose $\kappa$ is a submersion on $\Gamma \times \lac 0 \rac$, which implies $\Gamma$ is a submanifold of $\Cc^{2, \alpha}(\S^1)$ of codimension $1$.

The inclusion for $\Gamma$ is actually equivalent to the nullity of the vertical flux of an entire graph in $\Nil_3$. Consider a minimal surface admitting graph coordinates at infinity $X^{\eta}$ with $\eta \in \Cc^{2, \alpha}(\ov{\D})$. As shown in~\cite{Ca2}, the vertical flux $f_3$ of $X^{\eta}$ is for any $R \in (0, 1)$:
\begin{align*}
f_3 & = \int_0^{2\pi} \lr \Bigg{\lan} \frac{g_{22}(\eta)}{\sqrt{|g(\eta)|}} X^{\eta}_1- \frac{g_{12}(\eta)}{\sqrt{|g(\eta)|}} X^{\eta}_2, E_3 \Bigg{\ran} \rb_{r= R} d\theta \\
    & = \frac{4R^2}{(1+ R^2)^2} \int_0^{2\pi} \eta(R, \theta) d\theta+ O(1- R^2) \\
    & = \int_0^{2\pi} \eta|_{r= 1} d\theta,
\end{align*}
when taking the limit when $R \into 1$. Hence, if $X^{\eta}$ is an entire graph such that $\eta= \eta(\gamma, \lambda)$ for some data $\lac \gamma, \lambda \rac$ with $\gamma \in \Gamma$, the flux $f_3$ is zero and so is the mean of $\gamma$.
\end{proof}

Treibergs~\cite{Tr} showed that given a $\Cc^2$ curve $\gamma: \S^1 \into \R$, there exists a constant mean curvature complete entire vertical graph in $3$-dimensional Minkowski space which is asymptotically at signed distance $\gamma$ from the light cone. Proposition~\ref{prop:local} is actually a $\Cc^{2, \alpha}$ local version of this result in $\Nil_3$:

\begin{theorem} \label{thm:treibergs}
Consider $\gamma \in \Cc^{2, \alpha}(\S^1)$ small enough for the $\Cc^{2, \alpha}$-norm and with zero mean. Then there exists a minimal complete entire vertical graph at asymptotic horizontal signed distance $\gamma$ from $S^0$. Moreover, such a surface is unique up to vertical translations.
\end{theorem}

\begin{proof}
For any $\gamma$ sufficiently small in the $\Cc^{2, \alpha}$ norm, we have $\gamma \in \Gamma$ and $X^{\eta(\gamma, 0)}$ is a minimal entire graph admitting $\gamma$ as value at infinity. And as in Remark~\ref{rmk:lambda}, uniqueness comes from Daniel, Meeks and Rosenberg half-space theorem~\cite{DaMeRo}.
\end{proof}

\subsection{Periodic deformations} \label{subsec:perioddef}

In the sequel, we fix a natural $n \geq 2$ and a parameter $a> 0$. To ease the writing, we denote $\theta_n= \pi/ n$ and $\gamma_{k, u}$ the horizontal geodesic directed by $\cos(k\theta_n) E_1+ \sin(k\theta_n) E_2$ and passing through $(0, 0, u)$ for any $k \in \lac 0, \dots, n- 1 \rac$ and $u \in \R$. We are interested in rotationally symmetric solutions to act as barriers in Section~\ref{sec:saddle}. They are constructed by the following result, which a corollary of Theorem~\ref{thm:treibergs}:

\begin{proposition} \label{prop:symgraph}
There exist a minimal entire graph $S_n$ such that:
\begin{enumerate}
	\item\label{pt:reflection} For any $k \in \lac 0, \dots, n- 1 \rac$, $S_n$ contains the horizontal geodesic $\gamma_{k, 0}$ and is invariant for the geodesic reflection along that geodesic.
	\item\label{pt:positive} The height of $S_n$ is nonnegative on the (open) angular sector $\lac 0< \theta< \theta_n \rac$.
\end{enumerate}
\end{proposition}

\begin{proof}
For Point~\ref{pt:reflection}, fix $\epsilon> 0$ small enough such that $s_n: \theta \mapsto \epsilon\sin(n\theta)$ is in $\Gamma$. From Theorem~\ref{thm:treibergs}, we know there exists a minimal entire graph $S_n$ with $X^{\eta(s_n, \lambda)}$ as graph coordinates at infinity, where the translation parameter $\lambda$ is chosen so that $S_n$ contain the origin. Consider the map $\eta' \in \Cc^{2, \alpha}(\ov{D})$ defined for some $k \in \lac 0, \dots, n- 1 \rac$ by:
\[
\eta'(r, \theta)= -\eta(s_n, \lambda)(r, 2k\theta_n- \theta).
\]
The surface $X^{\eta'}$ is the image of $X^{\eta(s_n, \lambda)}$ under the geodesic reflection along $\gamma_{k, 0}$. The value at infinity of $\eta'$ is:
\[
\eta'|_{r= 1}(\theta)= -\eta(\gamma^n, \lambda)|_{r= 1}(2k\theta_n- \theta)= -\sin(2\pi- n\theta)= \sin(n\theta)= s_n(\theta),
\]
Using Remark~\ref{rmk:lambda}, we know that $X^{\eta'}$ and $X^{\eta(s_n, \lambda)}$ differ by a vertical translation and since $X^{\eta(s_n, \lambda)}$ contains the origin ---~which is fixed by the geodesic reflection along $\gamma_{k, 0}$~---, we get $X^{\eta'}= X^{\eta(s_n, \lambda)}$.

For Point~\ref{pt:positive}, consider the restriction of $S_n$ to the angular sector $\lac 0 \leq \theta \leq \theta_n \rac$, which we also denote $S_n$ in the sequel. For any $\theta \in [0, \theta_n]$, we have $s_n(\theta) \geq 0$ and $s_n$ is positive in the interior. Hence, the height of $S_n$ is bounded from below and by vertical translation $T_h$ with positive $h$, we can make $T_h(S_n)$ and $S^0$ disjoint on $\lac 0 \leq \theta \leq \theta_n \rac$. Now, suppose there exist a point inside the open angular sector at which the height of $S_n$ is negative. Translating back, we obtain a first contact point between $T_{h_0}(S_n)$ and $S^0$ for some $h_0> 0$, which is impossible by maximum principle.
\end{proof}

Given an open subset $G$ of $\R^p$, $p \geq 2$, a minimal hypersurface in Euclidean space $\R^{p+ 1}$ is said to be \emph{supported} on $G$ if it is the graph of a function that does not change sign over $G$ and is zero on the boundary $\partial G$. A question of Meeks and Rosenberg~\cite{MeRo} is to know if the number of disjoint domains supporting minimal graphs is bounded. Li and Wang~\cite{LiWa} proved it to be true in Euclidean space for any dimension and Tkachev~\cite{Tk} refined the bounds.

In $\Nil_3$, we consider open subsets of $\R^2$ ---~where $\R^2$ is seen as the range of the projection $\pi$~--- and vertical minimal graphs. The construction of surfaces $S_n$ shows:

\begin{theorem} \label{thm:disjgraphs}
In Heisenberg space, the number of disjoint domains supporting minimal vertical graphs is unbounded.
\end{theorem}

\section{Saddle towers in Heisenberg space} \label{sec:saddle}

In this section, we use notations introducted in Section~\ref{subsec:perioddef}. We build a Saddle tower with $2n$ ends distributed at constant angle $\theta_n$.

For any $b> 0$, consider the polygonal Jordan curve $\Gamma_b$ which is the reunion of the following geodesic segments:
\begin{gather*}
h_1(b)= \lac (t, 0, 0) \st{} 0 \leq t \leq b \rac, \quad h_2(b)= \lac (t\cos\theta_n, t\sin\theta_n, 0) \st{} 0 \leq t \leq b \rac, \\
\wt{h}_1(b)= \lac (t, 0, a) \st{} 0 \leq t \leq b \rac, \quad \wt{h}_2(b)= \lac (t\cos\theta_n, t\sin\theta_n, a) \st{} 0 \leq t \leq b \rac, \\
v_1(b)= \lac (b, 0, t) \st{} 0 \leq t \leq a \rac \quad \text{and} \quad v_2(b)= \lac (b\cos\theta_n, b\sin\theta_n, t) \st{} 0 \leq t \leq a \rac.
\end{gather*}
Note that $\wt{h}_1(b), \wt{h}_2(b)$ are horizontal lifts of $h_1(b), h_2(b)$ respectively.

Since the (euclidean) convex hull $H_b$ of $\Gamma_b$ is mean-convex, we know from Meeks and Yau~\cite{MeYa} that the Plateau problem with boundary $\Gamma_b$ is solvable, meaning there exists an embedded minimal disk $\Sigma_b \subset H_b$ bordered by $\Gamma_b$ (see Figure~\ref{fig:plateausol}).

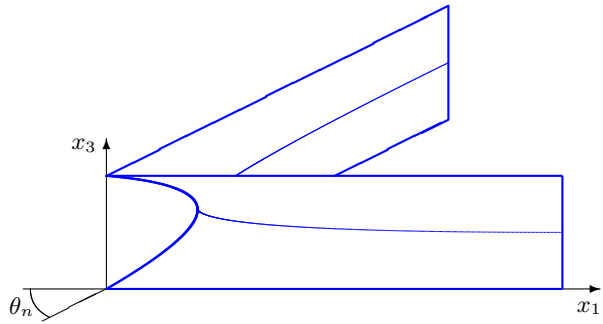
\begin{figure}[htbp] \centering
\begin{picture}(7.8, 4.2)(-1.3, -0.45)
\put(-1.1, 0){\vector(1, 0){7.6}} \put(0, 0){\vector(0, 1){2}}
\put(6.19, -0.3){\fns $x_1$} \put(-0.46, 1.85){\fns $x_3$}
\put(-0.84, -0.42){\line(2, 1){0.84}} \qbezier(-1, 0)(-1, -0.25)(-0.75, -0.375)
\put(-1.28, -0.33){\fns $\theta_n$}
\color{blue}
\qbezier(6, 0.75)(1.4, 0.75)(1.2, 1.05) \qbezier(4.5, 3)(2.5, 2.01)(1.7, 1.5)
\thicklines
\put(3, 1.5){\line(2, 1){1.5}} \put(0, 1.5){\line(2, 1){4.5}}
\put(0, 0){\line(1, 0){6}} \put (0, 1.5){\line(1, 0){6}}
\put(6, 0){\line(0, 1){1.5}} \put(4.5, 2.25){\line(0, 1){1.5}}
\qbezier(0, 0)(1.2, 0.7)(1.2, 1.05) \qbezier(0, 1.5)(1.2, 1.4)(1.2, 1.05)
\end{picture}
\caption{Plateau solution $\Sigma_b$ bordered by $\Gamma_b$ for some $b> 0$.} \label{fig:plateausol}
\end{figure}

To ensure the convergence of a subsequence of the family $(\Sigma_b)$ when $b \into +\ii$, we only need barriers from below and above and such barriers are the surfaces $S_n$ and $a- S_n$ respectively as constructed in Proposition~\ref{prop:symgraph}. Hence, there exists an embedded minimal surface $\Sigma_{\ii}$ bordered by $\Gamma_{\ii}$ which is the reunion of the horizontal geodesic rays:
\begin{gather*}
h_1= \lac (t, 0, 0) \st{} t \geq 0 \rac, \quad h_2= \lac (t\cos\theta_n, t\sin\theta_n, 0) \st{} t \geq 0 \rac, \\
\wt{h}_1= \lac (t, 0, a) \st{} t \geq 0 \rac \quad \text{and} \quad \wt{h}_2= \lac (t\cos\theta_n, t\sin\theta_n, a) \st{} t \geq 0 \rac.
\end{gather*}
The surface $\Sigma_{\ii}$ is the fundamental piece we are looking for. Extending $\Sigma_{\ii}$ by recursive geodesic reflections along the geodesics in its boundary.

\begin{theorem} \label{thm:saddle}
For any natural $n \geq 2$ and any $a> 0$, there exists a properly embedded minimal surface $\Ss(a, n)$ in $\Nil_3$ of genus zero, invariant by the rotation of angle $2\theta_n$ and axis $\lac x_1= x_2= 0 \rac$ and by the vertical translation of parameter $2a$. Moreover, for any $k \in \lac 0, \dots, n- 1 \rac$, $\Ss(a, n)$ contains the geodesics $\gamma_{k, 0}$ and $\gamma_{k, a}$ and is asymptotic ---~away from $\lac x_1= x_2= 0 \rac$~--- to the vertical plane containing $\gamma_{k, 0}$. We call $\Ss(a, n)$ a \emph{most symmetric Saddle tower}.
\end{theorem}

\appendix

\section{Proof of Theorem~\ref{thm:compcourb}} \label{sec:proof}

Consider a map $\eta \in \Cc^{2, \alpha}(\ov{\D})$ for some $\alpha \in (0, 1)$ ---~where $\Cc^{2, \alpha}(\ov{\D})$ denotes the usual H{\"o}lder space over $\ov{\D}$~--- and let $X^{\eta}: \D \into \Nil_3$ be the immersion given by:
\[
X^{\eta}(r, \theta)= \lp \frac{4r}{1- r^2} \cos \theta,  \frac{4r}{1- r^2} \sin \theta, \eta \frac{1+ r^2}{1- r^2} \rp,
\]
where $(r, \theta)$ are the polar coordinates on $\D$. Denoting by indexes $1, 2$ the derivatives with respect to $r, \theta$ respectively, the first derivatives of $X^{\eta}$ are:
\begin{gather*}
X^{\eta}_1(r, \theta)= \frac{4(1+ r^2)}{(1- r^2)^2} E_{\rho}+ \frac{4r}{(1- r^2)^2} \lp \eta+ \frac{(1+ r^2)\eta_1}{4r} (1- r^2) \rp E_3 \\
\text{and} \quad X^{\eta}_2(r, \theta)= \frac{4r}{1- r^2} E_{\theta}- \frac{8r^2}{(1- r^2)^2} \lp 1- \frac{(1+ r^2)\eta_2}{8r^2} (1- r^2) \rp E_3,
\end{gather*}
with $(E_{\rho}, E_{\theta}, E_3)$ denoting the cylindrical orthonormal frame in $\Nil_3$ and defined at generic point $(x_1= \rho \cos \theta, x_2= \rho \sin \theta, x_3)$ by:
\[
E_{\rho}= \cos \theta \der{x_1}+ \sin \theta \der{x_2}, \quad E_{\theta}= -\sin \theta \der{x_1}+ \cos \theta \der{x_2}+ \frac{\rho}{2} \der{x_3} \quad \text{and} \quad E_3= \der{x_3}.
\]

We then compute the first fundamental form $g(\eta)= \big{(} g_{ij}(\eta) \big{)}$:
\begin{gather*}
g_{11}(\eta)= \frac{16(1+ r^2)^2}{(1- r^2)^4} \lc \lp 1+ \frac{r^2 \eta^2}{(1+ r^2)^2} \rp+ \frac{r\eta \eta_1}{2(1+ r^2)} (1- r^2)+ \frac{\eta_1^2}{16} (1- r^2)^2 \rc, \\
g_{12}(\eta)= -\frac{32r^3}{(1- r^2)^4} \lc \eta+ \frac{1+ r^2}{4r} \lp \eta_1- \frac{\eta \eta_2}{2r} \rp (1- r^2)- \frac{(1+ r^2)^2 \eta_1 \eta_2}{32r^3} (1- r^2)^2 \rc \\
\text{and} \quad g_{22}(\eta)= \frac{16r^2(1+ r^2)^2}{(1- r^2)^4} \lc 1- \frac{\eta_2}{1+ r^2} (1- r^2)+ \frac{\eta_2^2}{16r^2} (1- r^2)^2 \rc.
\end{gather*}
And the determinant $|g(\eta)|$ of the first fundamental form writes:
\[
|g(\eta)|= \lp \frac{16r(1+ r^2)^2}{(1- r^2)^4} w(\eta) \rp^2= |g(0)|w^2(\eta),
\]
with the following expression of $w(\eta)$:
\begin{multline*}
w(\eta)= \Bigg{[} 1- \frac{\eta_2}{1+ r^2} (1- r^2)+ \lp \frac{r^2 \eta^2}{(1+ r^2)^4}+ \frac{\eta_2^2}{16} \rp (1- r^2)^2+ \frac{r\eta \eta_1}{2(1+ r^2)^3} (1- r^2)^3 \\
+\frac{\eta_1^2}{16(1+ r^2)^2} (1- r^2)^4 \Bigg{]}^{1/ 2}.
\end{multline*}
We also get the unit normal $N^{\eta}$ to $X^{\eta}$:
\begin{multline*}
N^{\eta}= \frac{1}{\sqrt{|g(\eta)|}} X^{\eta}_1 \times X^{\eta}_2= \frac{1}{w(\eta)} \Bigg{[} -\frac{r(1- r^2)}{(1+ r^2)^2} \lp \eta+ \frac{(1+ r^2)\eta_1}{4r} (1- r^2) \rp E_{\rho} \\
+\frac{2r}{1+ r^2} \lp 1- \frac{(1+ r^2) \eta_2}{8r^2} (1- r^2) \rp E_{\theta}+ \frac{1- r^2}{1+ r^2} E_3 \Bigg{]}.
\end{multline*}

Recall that the Levi-Civita connection $\del$ on $\Nil_3$ is given in the cylindrical frame $(E_{\rho}, E_{\theta}, E_3)$ by:
\[
\begin{array}{lll}
\del_{E_{\rho}} E_{\rho}= 0                       & \dst \del_{E_{\theta}} E_{\rho}= \frac{1}{\rho} E_{\theta}- \frac{1}{2} E_3 & \dst \del_{E_3} E_{\rho}= -\frac{1}{2} E_{\theta} \espf \\
\dst \del_{E_{\rho}} E_{\theta}= \frac{1}{2} E_3  & \dst \del_{E_{\theta}} E_{\theta}= -\frac{1}{\rho} E_{\rho}                 & \dst \del_{E_3} E_{\theta}= \frac{1}{2} E_{\rho}  \espf \\
\dst \del_{E_{\rho}} E_3= -\frac{1}{2} E_{\theta} & \dst \del_{E_{\theta}} E_3= \frac{1}{2} E_{\rho}                            & \del_{E_3} E_3= 0
\end{array}
\]
We compute the conormal derivatives:
\begin{gather*}
\del_{X^{\eta}_1} X^{\eta}_1= \frac{8r(3+ r^2)}{(1- r^2)^3} E_{\rho}- \frac{16r(1+ r^2)}{(1- r^2)^4} \lp \eta+ \frac{(1+ r^2)\eta_1}{4r} (1- r^2) \rp E_{\theta} \hspace{6.7em} \\
\hspace{10.1em}+ \frac{4}{(1- r^2)^3} \lp (1+ 3r^2)\eta+ 2r\eta_1(1- r^2) \rp E_3+ \frac{(1+ r^2)\eta_{11}}{1- r^2} E_3, \\
\del_{X^{\eta}_1} X^{\eta}_2= \frac{8r^2}{(1- r^2)^3} \lp \eta+ \frac{(1+ r^2)\eta_1}{4r} (1- r^2) \rp E_{\rho} \hspace{14.7em} \\
\hspace{18.7em}+ \frac{4(1+ r^2)^3}{(1- r^2)^4} \lp 1- \frac{\eta_2}{2(1+ r^2)} (1- r^2) \rp E_{\theta} \\
\hspace{11.3em}- \frac{8r(1+ r^2)}{(1- r^2)^3} \lp 1- \frac{\eta_2}{2(1+ r^2)} (1- r^2) \rp E_3+ \frac{(1+ r^2)\eta_{12}}{1- r^2} E_3 \\
\text{and} \quad \del_{X^{\eta}_2} X^{\eta}_2= -\frac{4r}{(1- r^2)^3} \lp (1+ 6r^2+ r^4)- (1+ r^2)\eta_2 (1- r^2) \rp E_{\rho}+ \frac{(1+ r^2)\eta_{22}}{1- r^2} E_3.
\end{gather*}
The Weingarten operator of $X^{\eta}$ is then determined by the following quantities:
\begin{multline*}
\lan \del_{X^{\eta}_1} X^{\eta}_1, N^{\eta} \ran= -\frac{32r^2}{w(\eta)(1- r^2)^4} \Bigg{[} \eta+ \frac{1+ r^2}{4r} \lp \eta_1- \frac{\eta\eta_2}{2r} \rp (1- r^2)- \frac{\eta_1\eta_2}{8r} (1- r^2)^2 \\
+ R'_{11}(1- r^2)^4 \Bigg{]}+ \frac{\eta_{11}}{w(\eta)},
\end{multline*}
with $R'_{11}= R'_{11}(r, \eta, D\eta)$ defined on $\ov{\D}$, identically zero if $\eta= 0$ and real-analytic in its variables. Also:
\begin{multline*}
\lan \del_{X^{\eta}_1} X^{\eta}_2, N^{\eta} \ran= \frac{32r^3}{w(\eta)(1- r^2)^4} \Bigg{[} 1- \frac{(1+ r^2)\eta_2}{4r^2} (1- r^2) \\
- \frac{1}{4} \lp \frac{\eta^2}{(1+ r^2)^2}- \frac{\eta_2^2}{4r^2} \rp (1- r^2)^2- \frac{\eta\eta_1}{8r(1+ r^2)} (1- r^2)^3+ R'_{12} (1- r^2)^4 \Bigg{]}+ \frac{\eta_{12}}{w(\eta)},
\end{multline*}
again with $R'_{12}= R'_{12}(r, \eta, D\eta)$ defined on $\ov{\D}$, zero if $\eta= 0$ and real-analytic in its variables, and:
\begin{multline*}
\lan \del_{X^{\eta}_2} X^{\eta}_2, N^{\eta} \ran= \frac{32r^4}{(1+ r^2)^2w(\eta)(1- r^2)^2} \Bigg{[} \eta+ \frac{1+ r^2}{4r} \lp \eta_1- \frac{\eta\eta_2}{2r} \rp (1- r^2) \\
+ R'_{22}(1- r^2)^4 \Bigg{]}+ \frac{\eta_{22}}{w(\eta)},
\end{multline*}
with $R'_{22}= R'_{22}(r, \eta, D\eta)$ defined on $\ov{\D}$, identically zero if $\eta= 0$ and real-analytic in its variables.

We shall now compute the mean curvature itself. Denote:
\[
H_{ij}(\eta)= g^{ij}(\eta) \lan \del_{X^{\eta}_i} X^{\eta}_j, N^{\eta} \ran,
\]
where $\big{(} g^{ij}(\eta) \big{)}$ is the inverse matrix of the metric $g(\eta)$. We get:
\begin{multline*}
H_{11}(\eta)= -\frac{2r^2}{(1+ r^2)^2 w^3(\eta)} \Bigg{[} \eta+ \frac{1+ r^2}{4r} \lp \eta_1- \frac{(1+ 10r^2+ r^4)\eta\eta_2}{2r(1+ r^2)^2} \rp (1- r^2) \\
- \frac{3\eta_2}{8r} \lp \eta_1- \frac{\eta\eta_2}{2r} \rp (1- r^2)^2+ \frac{\eta_2^2}{32r} \lp 3\eta_1- \frac{\eta\eta_2}{2r} \rp (1- r^2)^3 \\
+ R_{11}(1- r^2)^4 \Bigg{]}+ \frac{g^{11}(\eta)}{w(\eta)} \eta_{11},
\end{multline*}
with $R_{11}= R_{11}(r, \eta, D\eta)$ defined on $\ov{\D}$, identically zero if $\eta= 0$ and real-analytic in its variables, also:
\begin{multline*}
H_{12}(\eta)= \frac{4r^4}{(1+ r^2)^4 w^3(\eta)} \Bigg{[}\eta+ \frac{1+ r^2}{4r} \lp \eta_1- \frac{3\eta\eta_2}{2r} \rp (1- r^2) \\
- \frac{1}{4} \lp \frac{\eta^3}{(1+ r^2)^2}+ \frac{3\eta_1\eta_2}{2r}- \frac{3\eta\eta_2^2}{4r^2} \rp (1- r^2)^2 \\
- \frac{3}{16r} \lp \frac{\eta^2\eta_1}{1+ r^2}- \frac{\eta_1\eta_2^2}{2}- \frac{\eta^3\eta_2}{6r(1+ r^2)}+ \frac{\eta\eta_2^3}{12r} \rp (1- r^2)^3+ R_{12}(1- r^2)^4 \Bigg{]}+ \frac{g^{12}(\eta)}{w(\eta)} \eta_{12},
\end{multline*}
with $R_{12}= R_{12}(r, \eta, D\eta)$ defined on $\ov{\D}$, zero if $\eta= 0$ and real-analytic in its variables, and:
\begin{multline*}
H_{22}(\eta)= \frac{2r^2(1- r^2)^2}{(1+ r^2)^4 w^3(\eta)} \Bigg{[} \eta \lp 1+ \frac{r^2 \eta^2}{(1+ r^2)^2} \rp \\
+ \frac{r}{2} \lp \eta_1- \frac{\eta\eta_2}{2r}+ \frac{3\eta^2\eta_1}{2(1+ r^2)}- \frac{\eta^3\eta_2}{4r(1+ r^2)} \rp (1- r^2)+ R_{22}(1- r^2)^2 \Bigg{]}+ \frac{g^{22}(\eta)}{w(\eta)} \eta_{22},
\end{multline*}
with $R_{22}= R_{22}(r, \eta, D\eta)$ defined on $\ov{\D}$, zero if $\eta= 0$ and real-analytic in its variables. Finally, we obtain a Taylor expansion of the mean curvature:
\begin{align} \label{eq:taylor}
H(\eta) & = \frac{1}{2} \big{(} H_{11}(\eta)+ 2H_{12}(\eta)+ H_{22}(\eta) \big{)} \notag\\
        & = \frac{1}{2w(\eta)|g(\eta)|} \big{(} g_{22}(\eta)\eta_{11}- 2g_{12}(\eta)\eta_{12}+ g_{11}(\eta)\eta_{22} \big{)}+ R(1- r^2)^4,
\end{align}
as before with $R= R(r, \eta, D\eta)$ defined on $\ov{\D}$, identically zero if $\eta= 0$ and real-analytic in its variables.

Equation~\eqref{eq:taylor} can be written:
\begin{gather*}
H(\eta)= \frac{r}{\sqrt{|g(0)|}} \sum_{i, j} A_{ij}+ \frac{r}{\sqrt{|g(0)|}} B, \\
\text{with} \quad A_{11}= \frac{1}{2rw^3(\eta)} \frac{g_{22}(\eta)}{\sqrt{|g(0)|}}= \frac{1}{2}+ O(1- r^2) \\
A_{12}= A_{21}= -\frac{1}{2rw^3(\eta)} \frac{g_{12}(\eta)}{\sqrt{|g(0)|}}= \frac{\eta}{4}+ O(1- r^2) \\
\text{and} \quad A_{22}= \frac{1}{2rw^3(\eta)} \frac{g_{11}(\eta)}{\sqrt{|g(0)|}}= \frac{1}{2} \lp 1+ \frac{\eta^2}{4} \rp+ O(1- r^2)
\end{gather*}
Moreover, $A_{ij}= A_{ij}(r, \eta, D\eta)$ and $B= B(r, \eta, D\eta)$ are defined on $\ov{\D}$ and real-analytic in their variables, the matrix $A= (A_{ij})$ is coercive on $\ov{\D}$, and $B$ is zero when $\eta= 0$.

\newpage

\bibliographystyle{amsplain}

\providecommand{\bysame}{\leavevmode\hbox to3em{\hrulefill}\thinspace}
\providecommand{\MR}{\relax\ifhmode\unskip\space\fi MR }
\providecommand{\MRhref}[2]{%
  \href{http://www.ams.org/mathscinet-getitem?mr=#1}{#2}
}
\providecommand{\href}[2]{#2}
\begin{thebibliography}{}

\end{thebibliography}


\begin{thebibliography}{10}

\bibitem{AbRo2}
Uwe Abresch and Harold Rosenberg, \emph{Generalized {H}opf differential}, Mat.
  Contemp. \textbf{28} (2005), 1--28.

\bibitem{Ca2}
S{\'e}bastien Cartier, \emph{Noether invariants for constant mean curvature
  surfaces in {$3$}-dimensional homogeneous spaces}, preprint
  $\arXiv{:}1303.6391$, 2013.

\bibitem{CaHa}
S{\'e}bastien Cartier and Laurent Hauswirth, \emph{Deformations of constant
  mean curvature-{$1/ 2$} surfaces in {$\mathbb{H}^2 \times \mathbb{R}$} with
  vertical ends at infinity}, Comm. Anal. Geom. \textbf{22} (2014), no.~1,
  109--148.

\bibitem{ChSc}
Hyeong~In Choi and Richard Schoen, \emph{The space of minimal embeddings of a
  surface into a three-dimensional manifold of positive {R}icci curvature},
  Invent. Math. \textbf{81} (1985), no.~3, 387--394.

\bibitem{DaMeRo}
Beno{\^{\i}}t Daniel, William~H. Meeks, III, and Harold Rosenberg,
  \emph{Half-space theorems for minimal surfaces in {$\text{Nil}_3$} and
  {$\text{Sol}_3$}}, J. Differential Geom. \textbf{88} (2011), no.~1, 41--59.

\bibitem{GiTr}
David Gilbarg and Neil~S. Trudinger, \emph{Elliptic partial differential
  equations of second order}, Classics in Mathematics, Springer-Verlag, Berlin,
  2001, Reprint of the 1998 edition.

\bibitem{Gu}
Robert~D. Gulliver, II, \emph{Removability of singular points on surfaces of
  bounded mean curvature}, J. Differential Geometry \textbf{11} (1976), no.~3,
  345--350.

\bibitem{Kar}
Hermann Karcher, \emph{Embedded minimal surfaces derived from {S}cherk's
  examples}, Manuscripta Math. \textbf{62} (1988), no.~1, 83--114.

\bibitem{LiWa}
Peter Li and Jiaping Wang, \emph{Finiteness of disjoint minimal graphs}, Math.
  Res. Lett. \textbf{8} (2001), no.~5-6, 771--777.

\bibitem{MaTo}
Jos{\'e}~M. Manzano and Francisco Torralbo, \emph{New examples of constant mean
  curvature surfaces in {$\mathbb{S}^2 \times \mathbb{R}$} and {$\mathbb{H}^2
  \times \mathbb{R}$}}, preprint $\arXiv{:}1104.1259$, 2011.

\bibitem{Mar}
Valery~B. Marenich, \emph{Geodesics in {H}eisenberg groups}, Geom. Dedicata
  \textbf{66} (1997), no.~2, 175--185.

\bibitem{MaRoRo}
Laurent Mazet, M.~Magdalena Rodr{\'{\i}}guez, and Harold Rosenberg,
  \emph{Periodic constant mean curvature surfaces in {$\mathbb{H}^2 \times
  \mathbb{R}$}}, preprint $\arXiv{:}1106.5900$, 2011.

\bibitem{MeRo}
William~H. Meeks, III and Harold Rosenberg, \emph{The uniqueness of the
  helicoid}, Ann. of Math. (2) \textbf{161} (2005), no.~2, 727--758.

\bibitem{MeYa}
William~H. Meeks, III and Shing~Tung Yau, \emph{The classical {P}lateau problem
  and the topology of three-dimensional manifolds. {T}he embedding of the
  solution given by {D}ouglas-{M}orrey and an analytic proof of {D}ehn's
  lemma}, Topology \textbf{21} (1982), no.~4, 409--442.

\bibitem{Men}
Ana~M. Menezes, \emph{The {A}lexandrov problem in a quotient space of {$\mathbb{H}^2
  \times \mathbb{R}$}}, Pacific J. Math. \textbf{268} (2014), no.~1, 155--172.

\bibitem{Men3}
\bysame, \emph{Periodic minimal surfaces in semidirect products}, J. Aust.
  Math. Soc. \textbf{96} (2014), no.~1, 127--144.

\bibitem{MoRo}
Filippo Morabito and M.~Magdalena Rodr{\'{\i}}guez, \emph{Saddle towers and
  minimal {$k$}-noids in {$\mathbb{H}^2 \times \mathbb{R}$}}, J. Inst. Math.
  Jussieu \textbf{11} (2012), no.~2, 333--349.

\bibitem{Pi}
Ana~Lucia Pinheiro, \emph{Minimal vertical graphs in {H}eisenberg space},
  preprint, 2009.

\bibitem{Pl}
Julia Plehnert, \emph{New surfaces with genus one in {$\mathbb{H}^2 \times
  \mathbb{R}$}}, preprint $\arXiv{:}1212.2796$, 2012.

\bibitem{Py}
Juncheol Pyo, \emph{New complete embedded minimal surfaces in {$\mathbb{H}^2
  \times \mathbb{R}$}}, Ann. Global Anal. Geom. \textbf{40} (2011), no.~2,
  167--176.

\bibitem{RoTi}
M.~Magdalena Rodr{\'{\i}}guez and Giuseppe Tinaglia, \emph{Non-proper complete
  minimal surfaces embedded in {$\mathbb{H}^2 \times \mathbb{R}$}}, preprint
  $\arXiv{:}1211.5692$, 2012.

\bibitem{Rosen}
Harold Rosenberg, \emph{Minimal surfaces in {$\mathbb{M}^2 \times
  \mathbb{R}$}}, Illinois J. Math. \textbf{46} (2002), no.~4, 1177--1195.

\bibitem{Sc}
Heinrich~F. Scherk, \emph{Bemerkungen {\"u}ber die kleinste {F}l{\"a}che
  innerhalb gegebener {G}renzen}, J. Reine Angew. Math. \textbf{13} (1835),
  185--208.

\bibitem{Tk}
Vladimir~G. Tkachev, \emph{Disjoint minimal graphs}, Ann. Global Anal. Geom.
  \textbf{35} (2009), no.~2, 139--155.

\bibitem{Tr}
Andrejs~E. Treibergs, \emph{Entire spacelike hypersurfaces of constant mean
  curvature in {M}inkowski space}, Invent. Math. \textbf{66} (1982), no.~1,
  39--56.

\bibitem{Wh}
Brian White, \emph{The space of {$m$}-dimensional surfaces that are stationary
  for a parametric elliptic functional}, Indiana Univ. Math. J. \textbf{36}
  (1987), 567--603.

\end{thebibliography}
\providecommand{\bysame}{\leavevmode\hbox to3em{\hrulefill}\thinspace}
\providecommand{\MR}{\relax\ifhmode\unskip\space\fi MR }
\providecommand{\MRhref}[2]{%
  \href{http://www.ams.org/mathscinet-getitem?mr=#1}{#2}
}
\providecommand{\href}[2]{#2}

\vfill

\nidt S{\'e}bastien \textsc{Cartier}, Universit{\'e} Paris-Est, Laboratoire d'Analyse et de Math{\'e}matiques Appliqu{\'e}es (UMR 8050), UPEMLV, UPEC, CNRS, F-94010, Cr{\'e}teil, France \\
\textit{E-mail address:} \verb+sebastien.cartier@u-pec.fr+

\end{document}